\documentclass[11pt]{amsart}

\usepackage{amssymb,geometry,stmaryrd,color}
\usepackage{amsmath}
\usepackage{amsfonts}
\usepackage{fullpage}

\usepackage{hyperref}
\hypersetup{colorlinks,linkcolor={red},citecolor={blue},urlcolor={blue}}

\usepackage{graphicx}
\usepackage{caption}
\begin{document}

\newtheorem{theorem}{Theorem}[section]
\newtheorem{lemma}[theorem]{Lemma}
\newtheorem{proposition}[theorem]{Proposition}
\newtheorem{corollary}[theorem]{Corollary}
\newtheorem{conjecture}[theorem]{Conjecture}
\newtheorem{example}{Example}

\newtheorem{definition}[theorem]{Definition}
\newtheorem{remark}[theorem]{Remark}
\newtheorem{notation}[theorem]{Notation}
\newtheorem{question}[theorem]{Question}
\newtheorem{fact}[theorem]{Fact}

\numberwithin{equation}{section}

\def\s{{\bf s}} 
\def\t{{\bf t}} 
\def\u{{\bf u}} 
\def\x{{\bf x}} 
\def\y{{\bf y}} 
\def\z{{\bf z}} 
\def\B{{\bf B}} 
\def\C{{\bf C}} 
\def\D{{\bf D}}
\def\K{{\bf K}}
\def\F{{\bf F}}
\def\M{{\bf M}}
\def\ML{{\bf ML}}
\def\Nn{{\bf N}}
\def\G{{\bf \Gamma}} 
\def\W{{\bf W}}
\def\X{{\bf X}}
\def\U{{\bf U}}
\def\V{{\bf V}}
\def\Un{{\bf 1}}
\def\Y{{\bf Y}}
\def\Z{{\bf Z}}
\def\P{{\bf P}}
\def\Q{{\bf Q}}
\def\S{{\bf S}}
\def\L{{\bf L}}
\def\T{{\bf T}}

\def\cB{{\mathcal{B}}} 
\def\cC{{\mathcal{C}}} 
\def\cD{{\mathcal{D}}} 
\def\cG{{\mathcal{G}}} 
\def\cK{{\mathcal{K}}} 
\def\cL{{\mathcal{L}}} 
\def\cR{{\mathcal{R}}} 
\def\cS{{\mathcal{S}}}
\def\cU{{\mathcal{U}}}
\def\cV{{\mathcal{V}}} 
\def\cX{{\mathcal X}}
\def\cY{{\mathcal Y}}
\def\cZ{{\mathcal Z}}

\def\Ea{E_\a}
\def\eps{{\varepsilon}} 
\def\esp{{\mathbb{E}}} 
\def\Ga{{\Gamma}}

\def\lacc{\left\{}
\def\lcr{\left[}
\def\lpa{\left(}
\def\lva{\left|}
\def\racc{\right\}}
\def\rpa{\right)}
\def\rcr{\right]}
\def\rva{\right|}

\def\prst{{\leq_{st}}}
\def\prost{{\prec_{st}}}
\def\prcvx{{\prec_{cx}}}
\def\Rr{{\bf R}}

\def\CC{{\mathbb{C}}}
\def\EE{{\mathbb{E}}}
\def\NN{{\mathbb{N}}} 
\def\QQ{{\mathbb{Q}}} 
\def\PP{{\mathbb{P}}}
\def\ZZ{{\mathbb{Z}}}
\def\RR{{\mathbb{R}}}

\def\Tt{{\bf \Theta}}
\def\Ttt{{\tilde \Tt}}

\def\a{\alpha}
\def\A{{\bf A}}
\def\AA{{\mathcal A}}
\def\hAA{{\hat \AA}}
\def\hL{{\hat L}}
\def\hT{{\hat T}}

\def\claw{\stackrel{d}{\longrightarrow}}
\def\elaw{\stackrel{d}{=}}
\def\pslaw{\stackrel{a.s.}{\longrightarrow}}
\def\qed{\hfill$\square$}

\newcommand*\pFqskip{8mu}
\catcode`,\active
\newcommand*\pFq{\begingroup
        \catcode`\,\active
        \def ,{\mskip\pFqskip\relax}%
        \dopFq
}
\catcode`\,12
\def\dopFq#1#2#3#4#5{%
        {}_{#1}F_{#2}\biggl[\genfrac..{0pt}{}{#3}{#4};#5\biggr]%
        \endgroup
}

\def\ii{{\rm i}}

\title[ID of powers of positive Stable r.v.]{Infinite divisibility of powers of positive stable random variables}

\author[M.~Wang]{Min Wang}

\address{School of Mathematics and Statistics, Wuhan University of Technology,  Wuhan, 430070, China}

\email{minwangmath@whut.edu.cn}



\begin{abstract} 
   Let $Z_\alpha$ be a positive $\alpha$-stable random variable,
where $0<\alpha<1$. We prove that $Z_\alpha^p$ is infinitely
divisible for every $p>0$. Combining this result with Jedidi and Simon's characterization of infinite divisibility of negative powers of $Z_\alpha$, we obtain, for $p\ne0$,
$$Z_\alpha^p\ \text{is infinitely divisible}
\quad\Longleftrightarrow\quad
p\in\left(-\infty,-\frac{\alpha}{1-\alpha}\right]\cup(0,\infty).$$
The proof relies on Kanter's factorization and a general
product property: if $X$ is a mixture of exponential
distributions, $W$ is a generalized gamma convolution,
and $X$ and $W$ are independent, then $(X+c)W$ is infinitely
divisible for every $c\ge0$.
\end{abstract}

\maketitle

\section{Introduction}
 A real random variable $X$ is said to be stable if for any $a >0, b>0$, there exists $c >0, d \in \mathbb{R}$ such that 
\begin{equation}
\label{DefStable}
aX_1 + bX_2 \elaw c X +d,
\end{equation}
where $X_1, X_2$ are independent copies of $X$. The solution exists if and only if $c = (a^\alpha + b^\alpha)^{1/\alpha}$ for fixed $\alpha \in (0, 2]$. 
If $\a \neq 1$, \eqref{DefStable} is equivalent to 
$$a(X_1+r) + b(X_2+r) \elaw c (X+r), \quad r = \frac{d}{c-a-b}.$$
Except for the case $\alpha =1$ and $ d \neq 0$, every solution can be expressed as $Z_{\alpha,\rho}$,  up to a linear transformation, of which the characteristic function is 
$$E [ e^{it Z_{\alpha, \rho}} ]  = \exp [ - |t|^\alpha e^{i\pi\text{sgn}(t)(\frac{\alpha}{2} - \alpha\rho) } ], \quad t \in \mathbb{R},$$
where $\alpha \in (0,1], \rho \in [0,1]$ or $\alpha \in (1,2], \rho \in [1-1/\alpha, 1/\alpha]$. Here, $\rho$ is the positivity parameter, namely, $\rho = P[Z_{\alpha,\rho} \geq 0]$. 

In the present paper the infinite divisibility of $Z_\alpha^p$ is discussed, where  $Z_\alpha := Z_{\alpha,1},\, \alpha \in (0,1),$ is a positive $\alpha$-stable random variable and $p \in \mathbb{R}\setminus \{0\}$. 
Let's first recall the definitions of infinite divisibility, self-decomposability, generalized gamma convolutions and hyperbolically completely monotone densities. 
\subsection{Preliminaries}
A probability distribution on $[0, \infty)$ is called \emph{infinitely divisible} (ID) if, for every integer $n \ge 2$, it can be expressed as the convolution of $n$ identical probability distributions. Equivalently, its Laplace transform admits the %L\'evy--Khintchine 
representation
\begin{equation}
    \label{Levy-Khintchine}
    \mathbb{E}[e^{-\lambda X}] = \exp\left\{-a\lambda - \int_0^\infty (1 - e^{-\lambda x}) \,\nu(dx)\right\},
\end{equation}
where $a \ge 0$ and the L\'evy measure $\nu$ satisfies $\int_0^\infty (1 \wedge x) \,\nu(dx) < \infty$; see \cite[Theorem III.4.3]{SH04} or \cite[Theorem 3.2 and Theorem 5.9]{SSV10}. 

Several important subclasses arise by imposing additional structure on this representation.
A distribution is \emph{self-decomposable} (SD) if, for every $c \in (0, 1)$, one has
$$
X \stackrel{\mathrm{d}}{=} cX' + X_c,
$$
where $X' \stackrel{\mathrm{d}}{=} X$ and $X_c$ is independent of $X'$. On the nonnegative half-line, this property is equivalent to the representation \eqref{Levy-Khintchine} $\nu(dx) = k(x) \, dx/x$ with $k$ nonnegative and non-increasing. 
Requiring $k$ to be completely monotone, meaning that $(-1)^m k^{(m)}(x) \ge 0$ for every integer $m \ge 0$, yields the class of \emph{generalized gamma convolutions} (GGC). This class can also be characterized as the weak closure of finite convolutions of gamma distributions.

A more restrictive condition is hyperbolic complete monotonicity: a positive random variable has a \emph{hyperbolically completely monotone} (HCM) density $f$ if, for every $u > 0$, the product $f(uv)f(u/v)$ is completely monotone as a function of $w = v + v^{-1}$.

These classes satisfy the inclusions
$$
\text{HCM} \subset \text{GGC} \subset \text{SD} \subset \text{ID}.
$$

Another useful subclass of ID consists of \emph{mixtures of exponential distributions} (ME).
We allow mixtures of exponential distributions to have an atom
at zero. More precisely, a probability distribution belongs
to the class $\mathrm{ME}$ if it is the law of $\Gamma_1Y$,
where $\Gamma_1\sim\mathrm{Exp}(1)$ is independent of a
nonnegative random variable $Y$.
Equivalently, it is of the form
$$
\beta\delta_0+f(x)\,dx,
\qquad 0\le\beta\le1,
$$
where $f$ is completely monotone on $(0,\infty)$ and
$\beta+\int_0^\infty f(x)\,dx=1$. Its Laplace transform is a Stieltjes function with limit $1$
at zero; conversely, every such Stieltjes function is the
Laplace transform of a probability distribution in
$\mathrm{ME}$; see \cite[Definition 9.4 and Theorem 9.5]{SSV10}.

We refer to Bondesson \cite{Bon92}, Steutel and Van Harn \cite{SH04}, Sato \cite{Sat99} and Schilling-Song-Vondraček \cite{SSV10} for abundant properties of these subclasses.

\subsection{Results}
There are already some results in the literature. In 2013, Jedidi and Simon \cite[Proposition 2.1 and Theorem 2.1]{JS13} studied the property of negative powers of $Z_\alpha$ and proved that
\begin{itemize}
    \item[(1)] if $p < 0$, then $\, Z_\alpha^{p}$ is ID if and only if $p \leq - \frac{\alpha}{1-\alpha}$;
    \item[(2)] $\, Z_\alpha^{p}$ is SD if $\alpha \leq 1/2$ and $p \leq - \frac{\alpha}{1-\alpha}$;
    \item[(3)]  $\, Z_\alpha^{p}$ is GGC if $\alpha \in (0, 1/4]$ and $p \leq -4\alpha$. 
\end{itemize}

However, the ID property of positive powers of $Z_\alpha$ is not fully known. Patie \cite{Pat11} showed that  $Z_\alpha^\alpha$ is self-decomposable for every $\alpha \in (0,1)$. Bosch and Simon \cite[Section 3.2]{BS16} showed, without details of proof, that
$Z_\alpha^p$ is a generalized gamma convolution whenever
$p\ge\alpha/2$, and that it has an HCM density whenever
$\alpha\le1/2$ and $|p|\ge1$.
We recall these results and provide the proof in detail.

\begin{theorem}
\label{main theorem}
    The random variable $Z_\alpha^p$,
    \begin{enumerate}
        \item  is ID for all $\alpha \in (0,1)$ and $p > 0$. 
        \item  is GGC if $\alpha \in (0,1)$ and $p \geq \alpha/2$ or  $1/\alpha$ is an integer and $p > 0$.
        \item  is HCM if $\alpha \in (0,1/2]$ and $p \geq 1$.   
    \end{enumerate}
\end{theorem}

 Bosch and Simon \cite[page 631]{BS13} said:
``The present paper shows that all positive powers of $Z_{1/2}$ are SD, and one may wonder if the
same is true for $Z_\alpha$ with any $\alpha \in (0,1)$. "
This paper gives an affirmative answer on infinite divisibility. But the self-decomposability of $Z_\alpha^p$ is still open when $p \in (0, \alpha/2)$, $\alpha \in (0,1)$ and $1/\alpha$ is not an integer. 
 
 \begin{remark}
The proposition \ref{proposition ID closure} yields a stronger result.
Let $\mathrm{BO}$ denote the class of probability
distributions whose Laplace exponents are complete
Bernstein functions; see Chapter 9 in \cite{SSV10}. Since
$\mathrm{GGC}\subset\mathrm{BO}$,
$\mathrm{ME}\subset\mathrm{BO}$, and $\mathrm{BO}$ is
closed under independent addition, the decomposition
$$(X+c)W\stackrel d=c\widetilde W+M$$
shows that $(X+c)W\in\mathrm{BO}$.
Consequently, $Z_\alpha^p\in\mathrm{BO}$ for every
$0<\alpha<1$ and $p>0$.

Moreover, For $p\le-\alpha/(1-\alpha)$, Proposition 2.1 of
\cite{JS13} shows that $Z_\alpha^p$ has a completely
monotone density, hence belongs to
$\mathrm{ME}\subset\mathrm{BO}$. For $-\alpha/(1-\alpha)<p<0$, the same proposition
shows that $Z_\alpha^p$ is not infinitely divisible,
and therefore cannot belong to $\mathrm{BO}$. In conclusion, 
let $0<\alpha<1$ and $p\in\mathbb R\setminus\{0\}$, we have
$$Z_\alpha^p\in\mathrm{BO}
\quad\Longleftrightarrow\quad
Z_\alpha^p\in\mathrm{ID}
\quad\Longleftrightarrow\quad
p\in\left(-\infty,-\frac{\alpha}{1-\alpha}\right]\cup(0,\infty).$$

\end{remark}

\begin{notation}
\begin{enumerate}
    \item we denote the Gamma and Beta random variables by $\Gamma_{c}$ and  $B_{a, b}$, whose respective densities are
$$ \text{$ \frac{1}{\Gamma(c)}x^{c-1}e^{-x}\mathbf{1}_{(0, \infty)}(x) \quad $
and $\quad 
\frac{\Gamma(a+b)}{\Gamma(a)\Gamma(b)}x^{a-1}(1-x)^{b-1}\mathbf{1}_{(0, 1)}(x).$ }$$
    \item Throughout, unless otherwise explicitly stated, in any factorization of the type $X \elaw Y + Z$ or $X \elaw Y \times Z$, 
the random variables $Y, Z$ on the right-hand side will be assumed to be independent.
\end{enumerate}
\end{notation}

\textbf{Disclosure of AI assistance.}
The author used OpenAI's GPT-6 Astra to assist in proving Lemma~\ref{lemma continuous-stieltjes-mixtures} and to draw the figure at the end of the paper.

\section{Proof of Theorem \ref{main theorem} (1)} 
Recall that the fractional moments of a positive stable random variable are, see Zolotarev \cite[Theorem 2.6.3]{Zol86},
\begin{equation}
    E[Z_\alpha^{ps}] = \frac{\Gamma(1-ps/\alpha)}{\Gamma(1-ps)}, \quad p > 0, \;\; s < \alpha/p.
\end{equation}
We have the following identity:
\begin{equation}
    \label{eq decomposition}
    Z_\alpha^p \elaw \Gamma_1^{-\frac{1-\alpha}{\alpha}p} \times V_\alpha^{\frac{1-\alpha}{\alpha}p},
\end{equation}
where  $V_\a^{\frac{1-\a}{\a}}$  is the so-called Kanter random variable. Factorization \eqref{eq decomposition} was discovered by Kanter \cite[Corollary 4.1]{Kan75}. Afterwards, the Kanter random variable has been extensively studied in \cite{Dem11,JS13,HSW20}.
Theorem 4 in \cite{BS15} says $\Gamma_a^{-s}$ is GGC for every $a, s > 0$. We will also need Corollary 3.2 in \cite{JS13}, we rewrite it here for the readers' convenience.
\begin{lemma}[Corollary 3.2 in \cite{JS13}]
    \label{lemma CM property}
    For every $s > 0$, there exists $c_{\alpha,s} >0$ such that $V_\alpha^s -c_{\alpha,s}$ is \emph{ME}.
\end{lemma} 

Our proof is based on the observation that the above factorization \ref{eq decomposition} and Lemma \ref{lemma CM property} imply that every positive power of positive stable random variable has the form $(X + c)W$, where $c > 0$, $X$ is ME and $W$ is GGC. The following proposition completes the proof.

\begin{proposition}
    \label{proposition ID closure}
    Let $X$ and $W$ be independent nonnegative random variables,
with $X\in\mathrm{ME}$ and $W\in\mathrm{GGC}$.
For every $c\ge0$, there exist independent nonnegative
random variables $\widetilde W$ and $M$ such that
$$\widetilde W\stackrel d=W,\qquad
M\in\mathrm{ME},\qquad
(X+c)W\stackrel d=c\widetilde W+M.$$
In particular, $(X+c)W$ is infinitely divisible.
\end{proposition}

\begin{proof}[Proof of Proposition \ref{proposition ID closure}]
If $c=0$, $XW$ is again ME, therefore ID. We set $c >0$. 
Write $X\stackrel d=\Gamma_1Y$, where
$\Gamma_1\sim\mathrm{Exp}(1)$, and choose $\Gamma_1$, $Y$,
and $W$ to be independent. Define
$$L(s)=\mathbb E[e^{-sW}],\qquad
H(s)=\mathbb E[e^{-s(X+c)W}],\qquad
R(s)=\frac{H(s)}{L(cs)},\qquad s>0.$$
The denominator is strictly positive. We will prove that
$R$ is a Stieltjes function and that $R(0+)=1$.

Recall that a GGC is a weak limit of convolutions of a finite number of gamma distributions. As the GGC class is closed with respect to weak limits, it suffices to consider
the finite case where $W$ is a sum of $n$ gamma distributed variables. More precisely, we suppose that
\begin{equation}
    W \elaw a + \sum_{j=1}^n \frac{1}{b_j} \Gamma_{\alpha_j},
\end{equation}
where $a \geq 0, b_j > 0, \alpha_j > 0$ and all random variables are independent. Conditioning on $(Y,W)$ gives
$$H(s)=\mathbb{E}\left[\frac{e^{-csW}}{1+sYW}\right].$$
By direct computation, 
\begin{equation}
    L(s) = e^{-as}\prod_{j=1}^{n}\left(1+\frac{s}{b_j}\right)^{-\alpha_j}.
\end{equation}
%\newpage
\begin{align}
   &R(s) = \frac{H(s)}{L(cs)}\\
     %=&  \mathbb{E}\left[\frac{e^{-sc\sum_{j=1}^n \frac{1}{b_j} \Gamma_{\alpha_j}}}{1 + sY(a + \sum_{j=1}^n \frac{1}{b_j} \Gamma_{\alpha_j})} \frac{1}{ \prod_{j=1}^{n}\left(1+\frac{cs}{b_j}\right)^{-\alpha_j}} \right]   \\
     =& \mathbb{E}\left[ \int_0^\infty \cdots \int_0^\infty \frac{1}{ 1 + sY(a + \sum_{j=1}^n \frac{1}{b_j}u_j)}     \frac{e^{-sc\sum_{j=1}^n \frac{1}{b_j} u_j}}{ \prod_{j=1}^{n}\left(1+\frac{cs}{b_j}\right)^{-\alpha_j}} \prod_{j=1}^{n} \frac{1}{\Gamma(\alpha_j)} u_j^{\alpha_j -1} e^{-u_j} du_j  \right] \\
     & (\text{set} \;\; u_j = \frac{b_j}{b_j + cs} w_j ) \\
    = & \mathbb{E}\left[ \int_0^\infty \cdots \int_0^\infty \frac{1}{ 1 + sY(a + \sum_{j=1}^n \frac{w_j}{b_j + cs})} \prod_{j=1}^{n} \frac{1}{\Gamma(\alpha_j)} w_j^{\alpha_j -1} e^{- w_j}  dw_j \right]  \\
    = & \mathbb{E}\left[ \frac{1}{ 1 + sY(a + \sum_{j=1}^n \frac{\Gamma_{\alpha_j}}{b_j + cs})} \right].  
\end{align}
For fixed $y > 0$ and $u_j > 0$, the function 
\begin{equation}
\label{eq func1}
    s \mapsto \frac{1}{ 1 + sy(a + \sum_{j=1}^n \frac{u_j}{b_j + cs})} \quad \text{is a Stieltjes function.}
\end{equation}
 Indeed, by the definition of Stieltjes function (see Definition 2.1 in \cite{SSV10}), 
\begin{equation}
\label{eq func2}
    s \mapsto  \frac{1}{s} + y(a + \sum_{j=1}^n \frac{u_j}{b_j + cs}) \quad \text{is a Stieltjes function.}
\end{equation}
By Theorem 6.2 and Theorem 7.3 in \cite{SSV10}, we have the following equivalence:
\begin{center}
$f(s)/s$ is a Stieltjes function $\; \Longleftrightarrow \;$  $f $ is a complete Bernstein function 
\end{center}
\begin{center}
$\; \Longleftrightarrow \;$ $1/f $ is a Stieltjes function, 
\end{center}
where $f$ is a non-negative function on $(0,\infty)$, $f \not\equiv 0 $ and $1/f \not\equiv 0$. By this equivalence, \eqref{eq func1} can be derived from \eqref{eq func2}. 
We then finish the proof by the following lemma. 

\begin{lemma}%[Mixtures of continuously parametrized Stieltjes functions]
\label{lemma continuous-stieltjes-mixtures}
Let $d\ge1$, and let $\mu$ be a Borel probability measure on
$[0,\infty)^d$. Suppose that
$$
    K:(0,\infty)\times[0,\infty)^d\longrightarrow[0,1]
$$
satisfies the following conditions:
\begin{enumerate}
    \item For every $\theta\in[0,\infty)^d$, the function
    $s\mapsto K(s;\theta)$ is a Stieltjes function.
    \item For every $s>0$, the map
    $\theta\mapsto K(s;\theta)$ is continuous.
\end{enumerate}
Then
$$
    R(s):=\int_{[0,\infty)^d}K(s;\theta)\,\mu(d\theta),
    \qquad s>0,
$$
is a Stieltjes function.

Moreover, if $K(0+;\theta)=1$ for $\mu$-almost every $\theta$,
then $R(0+)=1$.
\end{lemma}

\begin{proof}
For each integer $m\ge1$, define
$$
    \pi_m(x):=
    \min\left\{m,\frac{\lfloor mx\rfloor}{m}\right\},
    \qquad x\ge0,
$$
and apply this map coordinatewise:
$$
    \Pi_m(\theta):=
    \bigl(\pi_m(\theta_1),\ldots,\pi_m(\theta_d)\bigr),
    \qquad
    \theta=(\theta_1,\ldots,\theta_d)\in[0,\infty)^d.
$$
Then $\Pi_m$ takes only finitely many values, and
$$
    \Pi_m(\theta)\longrightarrow\theta,
    \qquad\text{as }m\to\infty,
$$
for every $\theta\in[0,\infty)^d$.

Define
$$
    R_m(s):=
    \int_{[0,\infty)^d}
    K(s;\Pi_m(\theta))\,\mu(d\theta).
$$
Enumerate the distinct values of $\Pi_m$ as
$\theta_{m,1},\ldots,\theta_{m,N_m}$, and set
$$
    p_{m,k}:=
    \mu\bigl(\{\theta:\Pi_m(\theta)=\theta_{m,k}\}\bigr).
$$
Thus
$$
    R_m(s)=
    \sum_{k=1}^{N_m}p_{m,k}K(s;\theta_{m,k}),
    \qquad
    p_{m,k}\ge0,
    \qquad
    \sum_{k=1}^{N_m}p_{m,k}=1.
$$
Each function $K(\,\cdot\,;\theta_{m,k})$ is Stieltjes.
Since the Stieltjes class is a convex cone (see \cite[Theorem 2.2 (ii)]{SSV10}), $R_m$ is
a Stieltjes function.

Fix $s>0$. By continuity in the parameter,
$$
    K(s;\Pi_m(\theta))
    \longrightarrow K(s;\theta)
$$
for every $\theta\in[0,\infty)^d$.
Furthermore,
$$
    0\le K(s;\Pi_m(\theta))\le1.
$$
Since $\mu$ is a probability measure, bounded convergence
therefore yields
$$
    R_m(s)\longrightarrow
    \int_{[0,\infty)^d}K(s;\theta)\,\mu(d\theta)
    =R(s).
$$
The Stieltjes class is closed under finite pointwise limits; see \cite[Theorem 2.2 (iii)]{SSV10}.
Consequently, $R$ is a Stieltjes function.

Finally, if $K(0+;\theta)=1$ for $\mu$-almost every $\theta$,
another application of bounded convergence gives
$$
    R(0+)
    =\int_{[0,\infty)^d}K(0+;\theta)\,\mu(d\theta)
    =1.
$$
\end{proof}

Since GGC and ME are
infinitely divisible, their independent sum is infinitely
divisible. This completes the proof.

\end{proof}

\section{Proof of Theorem \ref{main theorem} (2)}
The main idea behind the proof of the GGC property is to represent a positive stable random variable as an infinite product of independent beta random variables. Then we use Theorem 1 in \cite{Bon15} which stated that if $X$ and $Y$ are two independent GGCs, then $XY$ is also a GGC. We call this property the closure of GGC class under independent multiplication. 
We need the following lemma.

\begin{lemma}
\label{lemma Beta inverse}
    Let $a, b, p > 0$. The random variable $B_{a,b}^{-p}$ is GGC if at least one of the following conditions holds:
    \begin{enumerate}
        \item $b > 1$ and $p \geq 1/2$; 
        \item $b$ is a positive integer;
        \item $p$ is a positive integer;       
        \item  $0 < b < 1$,  $p \geq 1/2$, and $a + b +p -\lfloor p\rfloor \geq 1$.  
    \end{enumerate}
    where $\lfloor p\rfloor$ is the integer part of $p$. 
\end{lemma}

\begin{remark}
\begin{itemize}
    \item[(a)]  Jedidi and Simon \cite[Conjecture 3.2]{JS13} conjectured that the random variable $B_{a,b}^{-p}$ is GGC for every $a, b > 0$ and every $p \geq 1$. Lemma \ref{lemma Beta inverse} partly confirms this conjecture.  
    The case $b < 1$, $p > 1$,  $p \notin \mathbb{N}$ and $a + b + p - [p] < 1$ is still open.
    \item[(b)] Bosch and Simon \cite[Section 3.1]{BS15} discussed the GGC property of $B_{a,b}^{-p}$, and asserted that for all $a > 0$, there are some $b \in (0,1), p \in (0,1)$ such that $B_{a,b}^{-p}$ is not a GGC. 
\end{itemize}
\end{remark}

\begin{proof}[Proof of Lemma \ref{lemma Beta inverse}]
    Firstly, by \cite[Theorem 2]{BS15}, $B_{a,b}^{-p}$ is GGC if one of the following three conditions holds.
    \begin{enumerate}
        \item $b > 1$ and $p >1$. 
        \item $b =1$ or $p = 1$.         
        \item  $b < 1$,  $p \in [1/2,1)$, and $a + b +p \geq 1$.  
    \end{enumerate}

    Secondly, Example VI.12.21 in Steutel and Van Harn \cite{SH04} stated that $- \log B_{a,b}$ is GGC if and only if 
    $b$ is an integer. Then the GGC property of  $B_{a,n}^{-p}$ can be deduced from Theorem 3 in Bondesson \cite{Bon15}, which said that if $X$ is a GGC with left extremity $a$, then $e^{X}-e^a$ is also a GGC. 

    Finally, we need to prove three cases: 
    \begin{itemize}
        \item[(i)] $b > 1$, $b \notin \mathbb{N}$ and $p \in [1/2, 1)$.
        \item[(ii)]  $p$ is an integer and $p \geq 2$.
        \item[(iii)]  $b < 1$,  $p > 1$, $p \notin \mathbb{N}$, and $a + b +p - [p]\geq 1$.
    \end{itemize}

Malmstén's formula for the Gamma function ($\psi$ is the digamma function), $-$see, for example, formula 1.9(1), page 21, in Erd\'elyi et al. \cite{EMOT53}
$$ \frac{\Gamma(z + a)}{\Gamma(a)} = \exp \lacc \psi(a)z + \int_0^\infty (e^{-zx} - 1 + zx)\frac{e^{-ax} dx}{x(1-e^{-x})}\racc, \,\, z > -a, \, a > 0, $$
entails the expression of fractional moments of Beta random variable
\begin{equation}
    \label{Malmsten Beta}
    E[ B_{a,b}^{s}] = \frac{\Gamma(a+s)\Gamma(a+b)}{\Gamma(a+b+s)\Gamma(a)} = \exp \left( - \int_0^\infty (1-e^{-sx}) \frac{e^{-ax} - e^{-(a+b)x}}{x(1-e^{-x})} dx \right), \quad s > -a. 
\end{equation}
We then deduce the following two identities in distribution
\begin{equation}
    \label{Id beta 1}
    B_{a, b+c} \stackrel{d}{=} B_{a,b} \times B_{a+b,c}, \quad a,b,c > 0,
\end{equation}
and 
\begin{equation}
    \label{Id beta 2}
    B_{a, b} \stackrel{d}{=}  \prod_{k=0}^n B_{\frac{k+a}{n+1},\frac{b}{n+1}}^{\frac{1}{n+1}} , \quad a, b > 0,  n \in \mathbb{N}, 
\end{equation}
by comparing their fractional moments. \\
In particular, by \eqref{Id beta 1} we have 
\begin{equation}
\label{Id beta 3}
    B_{a, b}^{-p} \stackrel{d}{=} B_{a,n}^{-p} \times B_{a+n,b-n}^{-p}, \quad 1 \leq n < b <n+1, n \in \mathbb{N}. 
\end{equation}
By \eqref{Id beta 2} we have 
\begin{equation}
\label{Id beta 4}
    B_{a, b}^{-p} \stackrel{d}{=}\prod_{k=0}^n B_{\frac{k+a}{n+1},\frac{b}{n+1}}^{-\frac{p}{n+1}}, \quad 1 \leq n < p \leq n+1, n \in \mathbb{N}. 
\end{equation}

In case (i), put $n=\lfloor b\rfloor$ and use
$$
B_{a,b}^{-p}\stackrel d=
B_{a,n}^{-p}B_{a+n,b-n}^{-p}.
$$
The first factor is a GGC because $n$ is a positive integer.
For the second factor, $0<b-n<1$ and
$(a+n)+(b-n)+p=a+b+p>1$, so the known criterion for
$1/2\le p<1$ applies.

In case (ii), write $p=N$, where $N\ge2$ is an integer.
The beta product identity gives
$$
B_{a,b}^{-N}\stackrel d=
\prod_{k=0}^{N-1}B_{(a+k)/N,b/N}^{-1}.
$$
Each factor is a GGC by the case $p=1$.

In case (iii), set $n=\lfloor p\rfloor$, $m=n+1$, and
$q=p/m$. Then $1/2<q<1$ and
$$
B_{a,b}^{-p}\stackrel d=
\prod_{k=0}^{m-1}B_{(a+k)/m,b/m}^{-q}.
$$
For every $k=0,\ldots,m-1$, we have $0<b/m<1$ and
$$
\frac{a+k}{m}+\frac bm+q
=\frac{a+b+p+k}{m}\ge1,
$$
where the last inequality follows from
$a+b+p-\lfloor p\rfloor\ge1$.
The same criterion therefore shows that every factor is
a GGC. In all three cases, the desired conclusion follows
from closure of the GGC class under independent multiplication.
%Combining \eqref{Id beta 3} with the closure of GGC class under independent multiplication, we can prove the case (i). Similarly, using \eqref{Id beta 4} we can prove the case (ii) and (iii). 
\end{proof}

We continue to prove Theorem \ref{main theorem} (2). 
For every $\alpha \in (0, 1)$, one has the a.s. convergent factorization; see Bosch and Simon \cite[page 37]{BS16}
    \begin{equation}
        \label{stablefactorizetion}
         Z_\alpha^{\alpha} \elaw e^{\gamma(1-\alpha)} \prod_{n=0}^{\infty} b_n B_{1+n/\alpha, 1/\alpha - 1}^{-1}, 
    \end{equation}
    where $\gamma$ is Euler’s constant, $\psi$ is the digamma function and $b_n = e^{\psi(1+\frac{n}{\alpha})-\psi(\frac{n+1}{\alpha})}$. Then the conclusion follows from the closure of GGC class under independent multiplication and Lemma \ref{lemma Beta inverse}.

\section{Proof of Theorem \ref{main theorem} (3) and some remarks}
The HCM class has the property that if $X$ is HCM, then $X^q$ is HCM for $|q| \geq 1$, see e.g. \cite[Proposition 4]{Bon15}. 
 Bosch and Simon \cite[Theorem, page 32]{BS16} proved that $\, Z_\alpha$ is HCM if and only if $\alpha \in (0, 1/2]$. Combining these two results, we can prove Theorem \ref{main theorem} (3).

\begin{remark}
   Bosch \cite{Bos15} proved that $Z_\alpha^p$ is not HCM if $\alpha \in (1/2, 1)$ or $|p| < \frac{\alpha}{1-\alpha}$. He asked whether $\alpha \in (0,1/2]$ and $|p| \geq \frac{\alpha}{1-\alpha}$ is sufficient for $Z_\alpha^p$ to be HCM. The answer is negative since
   Hasebe, Simon and Wang \cite[Remark 9]{HSW20} proved that $Z_{\alpha}^{-\frac{\alpha}{1-\alpha}}$ is not HCM for $\alpha < 1/5$. 
\end{remark}

\begin{remark} We consider $Z_\alpha^p$ with $\alpha \in (1/2,1)$ and $p \leq -\alpha/(1-\alpha)$. Jedidi and Simon \cite[Proposition 2.1]{JS13} proved that it is ID. The condition for it to be SD is not known. Yamazato \cite{Yam78} proved that every self-decomposable distribution has a unimodal density function. Simon \cite{Sim12} found the sufficient and necessary condition for the unimodality of $Z_\alpha^p$. More precisely, 
    $Z_\alpha^p$ is unimodal if and only if $(\alpha, p) \notin D$ where 
    $$ D = \{(\alpha, p)| \alpha > 1/2, \, -R(\alpha) < p < -\alpha\}, $$
    and $R(\alpha)$ has the following bounds
    $$1/(4-4\alpha) \leq R(\alpha) \leq (\alpha/\sin^2(\pi \alpha)) \wedge (\alpha/(1-\alpha)). $$
    This yields that  $Z_\alpha^p$ with $\alpha \in (1/2,1)$ and $p \leq -\alpha/(1-\alpha)$ is also unimodal. We leave its SD property for future research. 
\end{remark}

\begin{remark}
    In conclusion, the following figure gives a visual presentation of the ID and related properties of powers of positive stable random variables, i.e. $Z_\alpha ^p$. 
\end{remark}

\includegraphics[scale = 0.6]{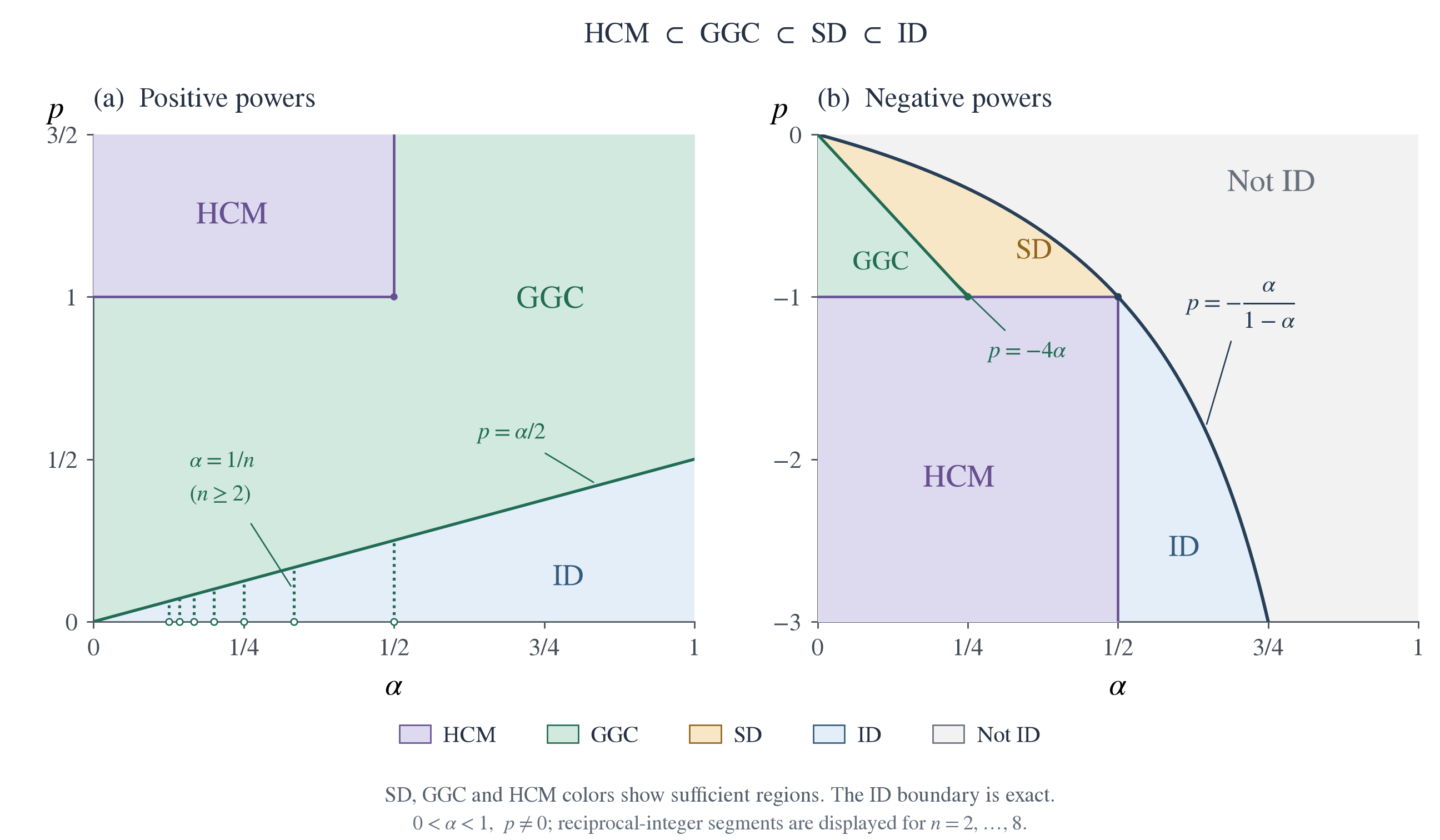}

\newpage

\bibliographystyle{plain} 
\bibliography{IDPowerStable}

@article {BS15,
    AUTHOR = {Bosch, Pierre and Simon, Thomas},
     TITLE = {On the infinite divisibility of inverse beta distributions},
   JOURNAL = {Bernoulli},
  FJOURNAL = {Bernoulli. Official Journal of the Bernoulli Society for
              Mathematical Statistics and Probability},
    VOLUME = {21},
      YEAR = {2015},
    NUMBER = {4},
     PAGES = {2552--2568},
      ISSN = {1350-7265,1573-9759},
   MRCLASS = {60E07 (33C05)},
  MRNUMBER = {3378477},
MRREVIEWER = {E.\ Sandhya},
       DOI = {10.3150/14-BEJ654},
       URL = {https://doi.org/10.3150/14-BEJ654},
}

@article {Pat11,
    AUTHOR = {Patie, P.},
     TITLE = {A refined factorization of the exponential law},
   JOURNAL = {Bernoulli},
  FJOURNAL = {Bernoulli. Official Journal of the Bernoulli Society for
              Mathematical Statistics and Probability},
    VOLUME = {17},
      YEAR = {2011},
    NUMBER = {2},
     PAGES = {814--826},
      ISSN = {1350-7265,1573-9759},
   MRCLASS = {60E07 (60G18 60G51 60J25)},
  MRNUMBER = {2787616},
       DOI = {10.3150/10-BEJ292},
       URL = {https://doi.org/10.3150/10-BEJ292},
}

@book {SH04,
    AUTHOR = {Steutel, Fred W. and van Harn, Klaas},
     TITLE = {Infinite divisibility of probability distributions on the real
              line},
    SERIES = {Monographs and Textbooks in Pure and Applied Mathematics},
    VOLUME = {259},
 PUBLISHER = {Marcel Dekker, Inc., New York},
      YEAR = {2004},
     PAGES = {xii+546},
      ISBN = {0-8247-0724-9},
   MRCLASS = {60-02 (60E07 60J10 60J80 60K05 60K25)},
  MRNUMBER = {2011862},
MRREVIEWER = {Rudolf\ Gr\"ubel},
}

@book {SSV10,
    AUTHOR = {Schilling, Ren\'e{} L. and Song, Renming and Vondracek, Zoran},
     TITLE = {Bernstein functions},
    SERIES = {De Gruyter Studies in Mathematics},
    VOLUME = {37},
   EDITION = {First},
      NOTE = {Theory and applications},
 PUBLISHER = {Walter de Gruyter \& Co., Berlin},
      YEAR = {2010},
     PAGES = {xiv+410},
      ISBN = {978-3-11-025229-3; 978-3-11-026933-8},
   MRCLASS = {60E07 (31C05 43A35 44A10 47A57 47D06 60E10 60Jxx)},
  MRNUMBER = {2978140},
MRREVIEWER = {David\ Applebaum},
       DOI = {10.1515/9783110269338},
       URL = {https://doi.org/10.1515/9783110269338},
}

@book {Sat99,
    AUTHOR = {Sato, Ken-iti},
     TITLE = {L\'evy processes and infinitely divisible distributions},
    SERIES = {Cambridge Studies in Advanced Mathematics},
    VOLUME = {68},
      NOTE = {Translated from the 1990 Japanese original,
              Revised by the author},
 PUBLISHER = {Cambridge University Press, Cambridge},
      YEAR = {1999},
     PAGES = {xii+486},
      ISBN = {0-521-55302-4},
   MRCLASS = {60G51 (60E07 60G18 60G52 60J45)},
  MRNUMBER = {1739520},
MRREVIEWER = {N.\ H.\ Bingham},
}

@article {JS13,
    AUTHOR = {Jedidi, Wissem and Simon, Thomas},
     TITLE = {Further examples of {GGC} and {HCM} densities},
   JOURNAL = {Bernoulli},
  FJOURNAL = {Bernoulli. Official Journal of the Bernoulli Society for
              Mathematical Statistics and Probability},
    VOLUME = {19},
      YEAR = {2013},
    NUMBER = {5A},
     PAGES = {1818--1838},
      ISSN = {1350-7265,1573-9759},
   MRCLASS = {60E05 (60E07 60E10 62E10)},
  MRNUMBER = {3129035},
       DOI = {10.3150/12-BEJ431},
       URL = {https://doi.org/10.3150/12-BEJ431},
}

@book {Bon92,
    AUTHOR = {Bondesson, Lennart},
     TITLE = {Generalized gamma convolutions and related classes of
              distributions and densities},
    SERIES = {Lecture Notes in Statistics},
    VOLUME = {76},
 PUBLISHER = {Springer-Verlag, New York},
      YEAR = {1992},
     PAGES = {viii+173},
      ISBN = {0-387-97866-6},
   MRCLASS = {60E07 (42A85 60E10 62E10)},
  MRNUMBER = {1224674},
MRREVIEWER = {F.\ W.\ Steutel},
       DOI = {10.1007/978-1-4612-2948-3},
       URL = {https://doi.org/10.1007/978-1-4612-2948-3},
}

@article {Bon15,
    AUTHOR = {Bondesson, Lennart},
     TITLE = {A class of probability distributions that is closed with
              respect to addition as well as multiplication of independent
              random variables},
   JOURNAL = {J. Theoret. Probab.},
  FJOURNAL = {Journal of Theoretical Probability},
    VOLUME = {28},
      YEAR = {2015},
    NUMBER = {3},
     PAGES = {1063--1081},
      ISSN = {0894-9840,1572-9230},
   MRCLASS = {60E07 (60E10)},
  MRNUMBER = {3413969},
MRREVIEWER = {Oliver\ Johnson},
       DOI = {10.1007/s10959-013-0523-y},
       URL = {https://doi.org/10.1007/s10959-013-0523-y},
}

@book {EMOT53,
    AUTHOR = {Erd\'elyi, Arthur and Magnus, Wilhelm and Oberhettinger, Fritz
              and Tricomi, Francesco G.},
     TITLE = {Higher transcendental functions. {V}ols. {I}, {II}},
      NOTE = {Based, in part, on notes left by Harry Bateman},
 PUBLISHER = {McGraw-Hill Book Co., Inc., New York-Toronto-London},
      YEAR = {1953},
     PAGES = {xxvi+302, xvii+396},
   MRCLASS = {33.0X},
  MRNUMBER = {58756},
MRREVIEWER = {E.\ T.\ Copson},
}

@article {Bos15,
    AUTHOR = {Bosch, Pierre},
     TITLE = {H{CM} property and the half-{C}auchy distribution},
   JOURNAL = {Probab. Math. Statist.},
  FJOURNAL = {Probability and Mathematical Statistics},
    VOLUME = {35},
      YEAR = {2015},
    NUMBER = {2},
     PAGES = {191--200},
      ISSN = {0208-4147,2300-8113},
   MRCLASS = {60E07 (60E05)},
  MRNUMBER = {3433648},
MRREVIEWER = {Gali\ Srilakshminarayana},
}

@article {BS13,
    AUTHOR = {Bosch, Pierre and Simon, Thomas},
     TITLE = {On the self-decomposability of the {F}r\'echet distribution},
   JOURNAL = {Indag. Math. (N.S.)},
  FJOURNAL = {Koninklijke Nederlandse Akademie van Wetenschappen.
              Indagationes Mathematicae. New Series},
    VOLUME = {24},
      YEAR = {2013},
    NUMBER = {3},
     PAGES = {626--636},
      ISSN = {0019-3577,1872-6100},
   MRCLASS = {60E07},
  MRNUMBER = {3064566},
MRREVIEWER = {E.\ Sandhya},
       DOI = {10.1016/j.indag.2013.04.006},
       URL = {https://doi.org/10.1016/j.indag.2013.04.006},
}

@article {Yam78,
    AUTHOR = {Yamazato, Makoto},
     TITLE = {Unimodality of infinitely divisible distribution functions of
              class {$L$}},
   JOURNAL = {Ann. Probab.},
  FJOURNAL = {The Annals of Probability},
    VOLUME = {6},
      YEAR = {1978},
    NUMBER = {4},
     PAGES = {523--531},
      ISSN = {0091-1798,2168-894X},
   MRCLASS = {60E05},
  MRNUMBER = {482941},
MRREVIEWER = {Stephen\ James\ Wolfe},
       URL =
              {http://links.jstor.org/sici?sici=0091-1798(197808)6:4<523:UOIDDF>2.0.CO;2-S&origin=MSN},
}

@article {Sim12,
    AUTHOR = {Simon, Thomas},
     TITLE = {On the unimodality of power transformations of positive stable
              densities},
   JOURNAL = {Math. Nachr.},
  FJOURNAL = {Mathematische Nachrichten},
    VOLUME = {285},
      YEAR = {2012},
    NUMBER = {4},
     PAGES = {497--506},
      ISSN = {0025-584X,1522-2616},
   MRCLASS = {60E07 (26A09 60E15)},
  MRNUMBER = {2899640},
MRREVIEWER = {Makoto\ Yamazato},
       DOI = {10.1002/mana.201000062},
       URL = {https://doi.org/10.1002/mana.201000062},
}

@article {HSW20,
    AUTHOR = {Hasebe, Takahiro and Simon, Thomas and Wang, Min},
     TITLE = {Some properties of the free stable distributions},
   JOURNAL = {Ann. Inst. Henri Poincar\'e{} Probab. Stat.},
  FJOURNAL = {Annales de l'Institut Henri Poincar\'e{} Probabilit\'es et
              Statistiques},
    VOLUME = {56},
      YEAR = {2020},
    NUMBER = {1},
     PAGES = {296--325},
      ISSN = {0246-0203,1778-7017},
   MRCLASS = {60E07 (46L54)},
  MRNUMBER = {4058989},
MRREVIEWER = {Octavio\ Arizmendi},
       DOI = {10.1214/19-AIHP962},
       URL = {https://doi.org/10.1214/19-AIHP962},
}

@book {Zol86,
    AUTHOR = {Zolotarev, V. M.},
     TITLE = {One-dimensional stable distributions},
    SERIES = {Translations of Mathematical Monographs},
    VOLUME = {65},
    EDITOR = {Silver, Ben},
      NOTE = {Translated from the Russian by H. H. McFaden},
 PUBLISHER = {American Mathematical Society, Providence, RI},
      YEAR = {1986},
     PAGES = {x+284},
      ISBN = {0-8218-4519-5},
   MRCLASS = {60-01 (60E07 60F05)},
  MRNUMBER = {854867},
       DOI = {10.1090/mmono/065},
       URL = {https://doi.org/10.1090/mmono/065},
}

@article {Dem11,
    AUTHOR = {Demni, Nizar},
     TITLE = {Kanter random variable and positive free stable distributions},
   JOURNAL = {Electron. Commun. Probab.},
  FJOURNAL = {Electronic Communications in Probability},
    VOLUME = {16},
      YEAR = {2011},
     PAGES = {137--149},
      ISSN = {1083-589X},
   MRCLASS = {60E07 (33C60 46L54 60B20)},
  MRNUMBER = {2783335},
MRREVIEWER = {Zhaozhi\ Fan},
       DOI = {10.1214/ECP.v16-1608},
       URL = {https://doi.org/10.1214/ECP.v16-1608},
}

@article {Kan75,
    AUTHOR = {Kanter, Marek},
     TITLE = {Stable densities under change of scale and total variation
              inequalities},
   JOURNAL = {Ann. Probability},
  FJOURNAL = {The Annals of Probability},
    VOLUME = {3},
      YEAR = {1975},
    NUMBER = {4},
     PAGES = {697--707},
      ISSN = {0091-1798},
   MRCLASS = {60E05},
  MRNUMBER = {436265},
MRREVIEWER = {Barthel\ W.\ Huff},
       DOI = {10.1214/aop/1176996309},
       URL = {https://doi.org/10.1214/aop/1176996309},
}

@article {BS16,
    AUTHOR = {Bosch, Pierre and Simon, Thomas},
     TITLE = {A proof of {B}ondesson's conjecture on stable densities},
   JOURNAL = {Ark. Mat.},
  FJOURNAL = {Arkiv f\"or Matematik},
    VOLUME = {54},
      YEAR = {2016},
    NUMBER = {1},
     PAGES = {31--38},
      ISSN = {0004-2080,1871-2487},
   MRCLASS = {60E10 (60E07)},
  MRNUMBER = {3475816},
MRREVIEWER = {Zbigniew\ J.\ Jurek},
       DOI = {10.1007/s11512-015-0216-0},
       URL = {https://doi.org/10.1007/s11512-015-0216-0},
}
\end{document}